\documentclass[smallextended]{svjour3}       
\usepackage[final]{changes}
\usepackage{graphicx}
\usepackage[utf8]{inputenc}
\usepackage{enumitem}
\usepackage{longtable}
\usepackage{amsmath}
\usepackage{amssymb}
\usepackage{xcolor}
\smartqed  
\journalname{TOP}
\begin{document}

\title{\replaced{An augmented filled function for global nonlinear integer optimization}{ A new algorithm for global nonlinear integer optimization}}

\author{Juan Di Mauro       \and
        Hugo D. Scolnik 
}

\institute{J. Di Mauro \at CONICET\\
              Departamento de Computaci\'on, Instituto de Ciencias de la Computaci\'on, Facultad de Ciencias Exactas y Naturales. UBA\\
              Pabell\'on I. Ciudad Universitaria, (C1428EGA) Buenos Aires, Argentina.\\
              Tel./Fax: (54.11) 5285-7438/7439/7440\\
              \email{jdimauro@dc.uba.ar}\\
              \and H. D. Scolnik \at 
              Departamento de Computaci\'on, Instituto de Ciencias de la Computaci\'on, Facultad de Ciencias Exactas y Naturales. UBA\\
              Pabell\'on I. Ciudad Universitaria, (C1428EGA) Buenos Aires, Argentina.\\
              Tel./Fax: (54.11) 5285-7438/7439/7440\\
              \email{hugo@dc.uba.ar}\\
}

\date{Received: date / Accepted: date}

\maketitle

\begin{abstract}
The problem of finding global minima of nonlinear discrete functions arises in many fields of practical matters. In recent years, methods based on discrete filled functions become popular as ways of solving these sort of problems. However, they rely on the steepest descent method for local searches. Here we present an approach that does not depend on a particular local optimization method, and a new discrete filled function with the useful property that a good continuous global optimization algorithm applied to it leads to an approximation of the solution of the nonlinear discrete problem (Theorem \ref{theorem:approx}). Numerical results are given showing the efficiency of the new approach.
\end{abstract}

 \section{Introduction}

\deleted{The discrete optimization of a nonlinear function, constrained to a box region is, in general, a problem hard to solve.}
\added{This paper is concerned with the analysis and performance of an algorithm designed to ”try to” solve model (\ref{problem}) efficiently.}
% \added{
\begin{equation}
\label{problem}
\min_{x\in X}f(x) 
\end{equation}
\[f:X\rightarrow \mathbb{R}; X=\{x\in\mathbb{Z}^n:a_i\leq x\leq b_i, i=1,\ldots,n\}\]
% }
\added{where $a_i$ , $b_i$ are, respectively, the lower and the upper bound of the variable $x_i$ ,
for $i = 1, \ldots, n$. We warn the reader that in general, no algorithm ensures con-
vergence to a global minimum. Nonetheless, we obtained global convergence on $77.5\%$
 of the preliminary numerical experiments reported in the appendix, on small problems.
}

\replaced{It is well known that discrete models are NP. Solving model (\ref{problem}) in particular has
shown to be a difficult task, even for polynomial functions with a few number of
variables \cite{adleman}}{Even in simple cases such as polynomials with a few variables, reaching the global minimum can be a task of high complexity as was shown in \cite{adleman}}. Moreover, the existence of multiple local minima may cause that an optimization algorithm stops at one of such minima, eventually giving minimizers of poor quality. 
	
	\replaced{Ways to}{The ways to} overcome the last issue include metaheuristics methods such as tabu search or simulated annealing and also exact methods as branch and bound, cutting planes or Lagrangian relaxation. 
	
	In recent years, a technique that makes use of an auxiliary function to escape from local minima, known as the filled functions approach, has gained attention.  

Ge \cite{ge1}, \cite{ge2} originally introduced the filled function method for continuous optimization. Later, Zhu  \cite{zhu1} carried that technique into the field of discrete optimization. Several discrete filled functions \replaced{have been proposed}{were invented} with one or more parameters and with additional features. However, in all  cases, the discrete steepest descent algorithm is employed in the search for a local minimizer. The use of that algorithm poses somewell known limitations in the effectiveness  of the optimization procedure. Even more, that choice conditions the definition of a filled function. For example, regarding a basin as a set of points that converges to a local minimum with the steepest descent algorithm. 
  
    Besides that, much of the efforts made over the years to have powerful continuous or discrete optimizations algorithms suggest that to constrain the definitions of a general method to a specific algorithm can be hardly considered as a reasonable approach. As will be shown, much can be gained preserving only the essential features of the process, while leaving other aspects unspecified, such as the local search procedure. 

    \replaced{However, some conditions must be imposed on the searching procedure to better
the performance of the filled functions approach}{However, certain conditions need to be assumed over that procedure expected to be satisfied if an optimization algorithm is good enough.}

    	As a consequence, new definitions are needed,  maintaining some level of accordance with the old ones.
     
     Moreover, a new filled function with some additional useful properties is desiderable. For instance,
   a useful result is that a good continuous global optimization algorithm applied to the new filled function gives an approximation to the discrete solution of the problem (see Theorem \ref{theorem:approx}).

          The main contribution of this paper is to present a new filled function which is independent of the chosen nonlinear optimization algorithm, \replaced{and allows the use of the best suitable method, avoiding that the evaluation of a filled function be based merely on its performance with elementary descent
algorithms.}{allowing in such a way to use the most powerful methods available nowadays like those based on trust regions, etc (see for instance \cite{nocedalbook}). In this sense, the performance problems observed in the classical filled functions approaches are mainly due to the use of the elementary gradient descent algorithm.}

	The paper is organized as follows: section \ref{sec:1} gives preliminary notions and notation as well as the new concepts. Also, key results relating the new definitions with the previous ones are provided. Section \ref{sec:2} introduces the concept of a filled function related to a general optimization algorithm. Section \ref{sec:3} shows one filled function verifying these definitions. Finally, section \ref{sec:4} shows computational experiments with test functions and compares the results with \replaced{others}{ other} in the literature.
 \section{Notation and Definitions} \label{sec:1}
 The set $\mathbb{Z}^n=\mathbb{Z}\times\mathbb{Z}\times\ldots\times\mathbb{Z}$ ($n$ times) is the set of the $n$-tuples $(x_1,\ldots,x_n)$ with $x_i\in\mathbb{Z},\ i=1,\ldots,n$.
 The vector $e_i\in\mathbb{R}^n$ is the elementary vector $i$, such that the $i$-th component is $1$ and all other entries are zero.

 If $x\in\mathbb{R}^n$, then $[x]$ is the point $x$ with rounded entries, that is

 \begin{equation}
   [x]_i= \Bigg\{ \begin{array}{cc}
                   \lfloor x_i+x_i/(2|x_i|)\rfloor & \textnormal{ if } x_i\notin \mathbb{Z}\\
                   x_i&\textnormal{ if } x_i\in \mathbb{Z}
                  \end{array}
\ i=1,\ldots,n 
 \end{equation}

 If $x$ is in $\mathbb{Z}^n$, $\mathcal{N}(x)$ is the discrete vicinity of $x$, \[\mathcal{N}(x)=\{x\pm e_i,\ i=1,\ldots,n\}\cup\{x\}.\] 
 
 \added{The set of directions in $\mathbb{Z}^n$ is }
 \[D=\{\pm e_i,\ i=1,\ldots,n\}\]
 
\deleted{The definitions for a basin, a discrete path, and others are the same as those in the literature (see, for example, \cite{ge1}).} 
\begin{definition}
 \added{A discrete local minimizer of $f$ is a point $x^*\in X$ such that $f(x^*)\leq f(x)$ for all $x\in \mathcal{N}(x^*)$. A discrete global minimizer of $f$ is a point $x^*\in X$ such that $f(x^*)\leq f(x)$ for all $x\in X$.}
\end{definition}

% \begin{definition}
%  \added{The \emph{usual} basin of $f(x)$ at an isolated minimizer $x^*$ is a connected domain $B^*$ which contains $x^*$ and in which starting from any point the steepest descent trajectory of $f(x)$ converges to $x^*$, but outside which the steepest descent trajectory of $f(x)$ does not converge to $x^*$.}
% \end{definition}
% A basin in that, typical sense, will be called \deleted{usual basin or }u-basin for short.
\added{Let $X=\{x\in\mathbb{Z}^n:a_i\leq x\leq b_i, i=1,\ldots,n\}$}
\begin{definition}
\added{A discrete path in $X$ between the points $x^*,x^{**}\in X$ is a sequence $\{x^{(i)}\}_{i=0}^n$ with}
\begin{enumerate}
 \item \added{$x^{(0)}=x^*$, $x^{(n)}=x^{**}$ }
 \item \added{$x^{(i)}\in X$ for all $i$}
 \item \added{$x^{(i)}\neq x^{(j)}$ for $j\neq i$}
 \item \added{$\|x^{(1)}-x^{*}\|=\|x^{(i+1)}-x^{(i)}\|=\|x^{**}-x^{(n-1)}\|=1$}
\end{enumerate}
\end{definition}
\added{The points $x^*$ and $x^{**}$, are said to be pathwise connected in $X$ if a discrete path in $X$ between them exists. If every two different points of a domain $X$ are pathwise connected in $X$ then $X$ is a pathwise connected domain or simply a connected domain.}
\begin{definition}
 \added{Let $x\in X$. A discrete descent direction of $f$ at $x$ over $X$ is $d\in D$ such that $x+d\in X$ and $f(x+d)<f(x)$. Let $D^*$ be the set of all descent directions of $f$ at $x$ over $X$. A discrete steepest descent direction of $f$ at $x$ is a descent direction $d^*$ such that $f(x+d^*)\leq f(x+d)$ for all $d\in D^*$}
\end{definition}

\added{We recall the discrete steepest descent algorithm to find a local minimizer for $f$ over $X$ starting at $x$  }
\begin{enumerate}[label*=Step \arabic*. ]
 \item \added{Choose an initial point $x\in X$.}
 \item \added{If $x$ is a discrete local minimizer of $f$ over $X$ then stop. Else, let $d^*$ be the discrete steepest descent direction of $f$ at $x$ over $X$.\label{algo:stepdesc:2}}
 \item \added{Set $x\leftarrow x+d^*$ and go to \ref{algo:stepdesc:2} }
\end{enumerate}

\begin{definition}
 \added{A discrete \emph{usual} basin $B^*\subset X$ of $f$ at $x^*$ is a connected domain which contains $x^*$ and all those $x\in X$ for which the discrete steepest descent algorithm for $f$ starting at $x$ converges to $x^*$}
\end{definition}

  \replaced{The notation $a \leftarrow b$ means, as usual, that $a$ takes the value of $b$, and $x \leftarrow \mathcal{C}(f, x_0)$
means that $x \in \mathbb{R}^n$ is the value returned by algorithm $\mathcal{C}$ applied to a function
$f : \mathbb{R}^n \rightarrow \mathbb{R}$ bounded from below, starting at $x_0 \in \mathbb{R}^n$ . We must impose certain
assumptions to $\mathcal{C}$:}{As always, $a \leftarrow b$ means that $a$ takes the value of $b$.
  Although we do not want to constrain our definitions to a particular local search algorithm, some assumptions about $\mathcal{C}$ must be made. However, those assumptions are expected to be satisfied by most of the local search procedures.
  $\mathcal{C}$ will be a continuous, deterministic optimization algorithm that takes a function $f:\mathbb{R}^n\rightarrow \mathbb{R}$ bounded from below,
  an initial point $x_0\in\mathbb{R}^n$ and returns $x\in\mathbb{R}^n$ with the following properties:}
  \begin{enumerate}[label=\textbf{A\arabic*}]
  \item \added{$\mathcal{C}$ is deterministic }
   \item  {If $x\leftarrow \mathcal{C}(f,x_0)$ then $f(x)\leq f(x_0)$ \label{hipC1}}
   \item  {If $x_0$ is in $\mathbb{Z}^n$ and $x\leftarrow\mathcal{C}(f,x_0)$, then there is $d\in\{\pm e_i: i=1,\ldots,n\}$ such that $x\leftarrow\mathcal{C}(f,x_0+d)$
   and $\|x_0+d-x\|< \|x_0-x\|$ 
   or $x_0=x$. \label{hipC3}}
  \end{enumerate}
  
  \begin{definition}
  
  A basin $B^*$ of $f$ at $x^*$, a local minimizer of $f$ (not necessarily a discrete local minimum) is the set of all points which converge to $x^*$ with $\mathcal{C}$, that is:
  \[B^*=\{x: x^*\leftarrow \mathcal{C}(f,x)\}.\]
  \end{definition}
  \begin{definition}
  
  A discrete basin $B_e^*$ of $f$ at $x^*$, a discrete local minimizer of $f$ (but not necessarily a local minimum) is the set
  \[B_e^*=\{x\in\mathbb{Z}^n: x'\leftarrow \mathcal{C}(f,x)\textnormal{ and }[x']=x^*\}.\]\label{def:cuencadisc}
  \end{definition}
  From hereafter the concepts of basin or u-basin will be understood as the discrete versions of them.
  \\
   
  The hypothesis \ref{hipC3} guarantees an essential property of the basins:
  \begin{theorem}
   A discrete basin is a connected discrete domain.
  \end{theorem}
  \begin{proof}
  Let $B^*$ be a basin of $f$ at $x^*$ and $x',x''\in B^*$. By \ref{hipC3} there are discrete paths $\{x'=x'_1,\ldots,x'_m=x^*\}$ and $\{x''=x''_1,\ldots,x''_{m-1},x''_m=x^*\}$ with all points in $B^*$.
  The path $\{x'=x'_1,\ldots,x^*,x''_{m-1},\ldots,x''_1=x''\}$ is a discrete path and has all its points 
  in $B^*$.
  \end{proof}

  It is important to point out that if $\mathcal{C}$ is a \textit{good algorithm} then any point that converges to $x^*$ using the steepest descent,
  converges with $\mathcal{C}$ to a point at least as good as $x^*$.  That justifies the last hypothesis over $\mathcal{C}$.
 
  \begin{enumerate}[label=\textbf{A\arabic*},resume]
  \item \label{hipC4} If $U^*$ is an u-basin of $f$ at $x^*$, and $x$ is in $U^*$, then $x$ is in $B^*$, a basin of $f$
  at $x^*_b$ with $f(x^*_b)\leq f(x^*)$.
  \end{enumerate}

  The order relation for the basins is the same as in the case of u-basins. 
   Namely, if $B^*$ and $B^{**}$ are two basins of $f$ at $x^*$ and $x^{**}$ respectively, then $B^{**}$ is lower than $B^*$ if $f(x^*)\leq f(x^{**})$ (and higher if $f(x^*)> f(x^{**})$)
  Moreover since $\mathcal{C}$ is deterministic, if $B^*$ and $B^{**}$ are different
  then $B^* \cap B^{**}=\emptyset$.

  \begin{definition}
  Given $x^*$, a discrete local minimizer of $f:X\rightarrow\mathbb{R}$, $X\subset \mathbb{Z}^n$ and $B^*$ a discrete
  basin of $f$ at $x^*$, $F:X\rightarrow\mathbb{R}$ is a discrete filled function of $f$ at $x^*$ if it satisfies the following:
   \begin{enumerate}[label=\textbf{D\arabic*}]
   \item \label{d1} $x^*$ is a strict (discrete) local maximizer of $F$ over $X$.
   \item \label{d2} 
   \replaced{$F$ has discrete local minimizers neither in $B^*$ nor in any}
   {$F$ has not discrete local minimizers in $B^*$ or in any} basin of $f$ higher than $B^*$. 
   
   \item \label{d3} If $f$ has a basin $B^{**}$ in $x^{**}$ lower than $B^*$, then there is a point $x'\in B^{**}$ that minimizes $F$
   in the discrete path $\{x^*,\ldots,x',\ldots,x^{**}\}$ in $X$.
   \end{enumerate}
   \label{def:filleddisc}
   \end{definition}
   
   The following theorem shows that Definition \ref{def:cuencadisc} preserves the properties
   of the discrete filled functions.
    \begin{theorem}
    
    A discrete filled function with u-basins, satisfies the conditions \ref{d1}, \ref{d2} and \ref{d3} with the Definition \ref{def:cuencadisc} of a basin.
    \end{theorem}
    \begin{proof}

     The first condition does not depend upon the definition of a basin, so there is no need to prove anything for \ref{d1}. 
     \\
     
     For the condition \ref{d2}, assume that $B^{**}$ and $B^*$ are two distinct basins of $f$ at $x^{**}$ and $x^*$
     respectively and, by contradiction, $B^{**}$ is higher than $B^*$ and $x'\in B^{**}$ is a local minimizer of $F$.
     
     Let $U_1$ be a u-basin of $f$ at $x_1$ and $x'\in U_1$. 
	 Considering the u-basin $U^*$ of $f$ at the discrete local minimizer $x^*$, by \ref{d2} the discrete local minimizer $x'$ of $F$ cannot be in an u-basin higher than $U^*$,
     therefore $f(x_1)\leq f(x^*)$. 
     
     But then,
     by \ref{hipC4} $x'\in B_1$, and $B_1$ is a basin of $f$ at $x'_1$ with
     \[f(x'_1)\leq f(x_1) \leq f(x^*) < f(x^{**}).\]
     So $x'\in B_1 \cap B^{**}$ but $B_1 \cap B^{**}=\emptyset$ because the basins $B_1$ and $B^*$ are different.
     
     It remains to show that $F$ cannot have a local minimizer in $B^*$. But, if $x'$ is a local minimizer of $F$ in $B^*$, then
     $x'$ belongs to an u-basin $U_1$, which, by \ref{hipC4} is higher than $U^*$ or is $U^*$ contradicting \ref{d2} 
     due to the definition of an u-basin.
    \\
          
     For the last condition, suppose that $x^*$ is not a global minimizer of $f$ and let $\{x^*,\ldots,x',\ldots,x_u^{**}\}$
     be the discrete path
     in \ref{d3} according to u-basins.
     Now we have to prove that there is a discrete path $\{x^*,\ldots,x'',\ldots,x^{**}\}$ with $x''\in B^{**}$ the 
     minimizer of $F$ in that path and $B^{**}$ a basin of $f$ in $x^{**}$ lower than $B^*$, the basin o $f$ at $x^*$.
     
     The proof goes by cases:
     \begin{enumerate}[label=\roman*.]
     \item $x'$ is a discrete local minimizer of $F$:\label{i}
     
     Then $x'$ is in $B^{**}$, a basin of $f$ in $x^{**}$ lower than $B^*$, by the previous condition, \ref{d2}.
     It is enough to extend the path $\sigma_1=\{x^*,\ldots,x'\}$ with $\sigma_2=\{x_1,\ldots,x^{**}\}$ a discrete path
     in $B^{**}$ with $x_1\in\mathcal{N}(x')$, and take the
     element that minimizes $F$ in the set $\{x',x_1,\ldots,x^{**}\}$ as the minimizer of $F$ in the path
     $\sigma_1\sigma_2=\{x^*,\ldots,x',x_1,\ldots,x^{**}\}$.
     The existence of $\sigma_2$ is guaranteed since a basin is a connected set.
     \item The path $\sigma=\{x^*,\ldots,x',\ldots,x_u^{**}\}$ is in the same basin $B^*$, and $x'$ is not a local minimizer of $F$:\label{ii}
     
	It is enough to take $x'$ and a descent
	path $\sigma_1=\{x',x_1,\ldots,x_n\}$ of $F$ with $x_n$ a discrete local minimizer of $F$.  
	By the previous condition, $x_n$ is in a basin $B^{**}$ of $f$ in $x^{**}$, lower than $B^*$.
	It suffices to concatenate the path $\sigma_1$ with some path $\sigma_2=\{x_{n+1},\ldots,x^{**}\}$ in
	$B^{**}$ and choose the minimizer
	of $F$ in the set $\{x_n,x_{n+1},\ldots,x^{**}\}$ as the minimizer of $F$ in the path 
	$\{x^*,\ldots,x',\ldots,x_n,\ldots,x^*\}$.
     \item The path has points in at least two basins and $x'$ is not a local minimizer of $F$:\label{iii}
      
      If $\{x^*,\ldots,x',\ldots,x_u^{**}\}$ is 
      \[\{x^*,x_1,\ldots,x_n,x_{n+1}\ldots,x_m,\ldots,x_u^{**}\}\]
      with $\sigma_1=\{x^*,x_1,\ldots,x_n\}$ a path in $B^*$,
     $\sigma_2=\{x_{n+1},\ldots,x_m\}$ a path in $B^{**}$, $B^{**}$ a basin of $f$ lower than $B^*$ and $x'\in\{x_{n+1},\ldots,x_m\}$,
     then it suffices to extend the path $\sigma_2$ with $\sigma'=\{x'_{m+1},\ldots,x^{**}\}\in B^{**}$ a path in $B^{**}$
     and take the minimizer of $F$ in the set $\{x_{n+1},\ldots,x_m,x'_{m+1},\ldots,x^{**}\}$ as the minimizer of $F$ in the
     path 
     \[\{x^*,x_1,\ldots,x_n,x_{n+1}\ldots,x_m,x'_{m+1},\ldots,x^{**} \} \in B^{**}.\] 
      Otherwise, if $x'\in \sigma_1$, as $\sigma_1$ has all its elements in the same basin, case \ref{ii} holds.	
     \end{enumerate}
    \end{proof}
    
    \replaced{For instance, it is proved in \cite{newfilledYang2} that (\ref{eq:filled}) is a filled function for any target function $f$ and adequate value of the parameter $r>0$ provided that the problem is only box-constrained}
    {A filled function is usually constructed as an auxiliary function from the target function. As an example, if the target function is $f$, it is proved in \cite{discretefilledng2} that 
%     \begin{equation}
%      F_{\mu,\rho}(x;x^*)=f(x^*)-\min[f(x),f(x^*)]-\rho \| x-x^*\|^2+\mu\{\max[0,f(x)-f(x^*)]\}
%     \end{equation}
    is a filled function of $f$ for adequates values of $\mu$ and $\rho$.}
    \begin{equation}
    \label{eq:filled}
     F_{r,x*}(x)=\Big(\frac{1}{\|x-x^*\|^2+1}+1\Big)h\Big(h_r(f(x)-f(x^*))+\sum _{i=1}^m h_r(g_i(x)-r)\Big)
    \end{equation}
    \added{where } 

    \[h_r(t)=\Bigg\{ \begin{array}{cc}
                      0,&t\leq -r\\
                      \frac{r-2}{r^3}t^3+\frac{2r-3}{r^2}t^2+t+1,&-r< t\leq 0\\
                      t+1,&t>0
                      \end{array}
\]
   \[h(t)=\Bigg\{ \begin{array}{cc}
                      0,&t\leq \frac{1}{2}\\
                      -16t^3+36t^2-24t+5,&\frac{1}{2}< t\leq 1\\
                      1,&t>1
                      \end{array}
\]
    \added{In a region}
    \[S=\{x\in X\mid g_i(x)\leq 0,i=1,\ldots,m\}\]
%     \added{of the box}
%     \[X=\{x\in\mathbb{Z}^n:a_i\leq x\leq b_i, i=1,\ldots,n\}\]
    \added{(and when $S=X$ the sum involving $g_i$ is eliminated).}
    
    From $x^*$ a local minimizer of $f$, it is expected that the minimization of the filled function give a new point $x'$ which is not necessarily a local minimizer of $f$ but can be used as initial point for a minimization algorithm of $f$ and as a result of that, a local minimizer $x^{**}\neq x^{*}$ with $f(x^{**})\leq f(x^{*})$ will be found.

    \subsection{Generic Algorithm for Discrete Filled Functions}

    The generic algorithm used in the optimization process with filled functions \deleted{here} is the following:
    
\begin{enumerate}[label*=Step \arabic*. ]
 \label{genericalgo}
 \item Initialization. 
 
  Choose a starting point $x_0\in X$. Let $q=2n$. Set the bounds of each parameter of the filled function $F$. Initialize the parameters.\label{algo:s1}
  \item Local minimization of $f$.\label{algo:s2}
  \begin{enumerate}[label=\roman*.]
   \item Do $x'\leftarrow \mathcal{C}(f,x_0)$,
   \item $x^*\leftarrow argmin_{x\in\mathcal{N}([x'])}(f(x))$. 
   \end{enumerate}
  \item Neighborhood search. \label{algo:s3}
  \begin{enumerate}[label=\roman*.]
   \item Let $\mathcal{N}(x^*)\setminus\{x^*\}=\{x_1,\ldots,x_q\}$, $\ell\leftarrow 1$ \label{algo:s31}
   \item Define $x_c\leftarrow x_\ell$\label{algo:s32}
  \end{enumerate}
  \item Local minimization of $F$.\label{algo:s4}
  \begin{enumerate}[label=\roman*.]
   \item Do $x'_c\leftarrow \mathcal{C}(F,x_c)$,\label{algo:s41}
   \item $x'\leftarrow argmin_{x\in\mathcal{N}([x'_c])}(f(x))$. \label{algo:s42}
   \end{enumerate}
  \item Checking the status of $x'$.\label{algo:s5}.
  
   If $f(x')<f(x^*)$, set $x_0\leftarrow x'$ and go to \ref{algo:s2}.
   \item Checking other search directions.
   
    Adjust the parameters of the filled function $F$. If $x'$ is not a vertex in $X$ go to \ref{algo:s4}. Else, set $\ell \leftarrow \ell+1$.
    If $\ell \leq q$, go to \ref{algo:s3}\ref{algo:s32} 
    If the parameters of $F$ exceed their bounds, take $x^*$ as the global minimizer.
\end{enumerate}
 \section{Filled function with respect to an algorithm}
\label{sec:2}
The usual definition of a discrete filled function assumes that the local search is made using the steepest descent method. It will be advantageous that the definition does not rely upon a particular algorithm because in such a way more powerful local search methods can be employed.
\begin{definition}

	Let $x^*$ be a local minimizer of $f:X\rightarrow\mathbb{R}$, $B^*$ the basin of $f$ at $x^*$ and $\mathcal{C}$ a deterministic optimization algorithm that satisfies the hypothesis \ref{hipC1}-\ref{hipC4}. Then $F$ is said to be a filled function of $f$ at $x^*$ \textbf{with respect to} $\mathcal{C}$ if:
\begin{enumerate}[label=\textbf{DC}\arabic*]
 \item\label{dc1}: $x^*$ is a strict local maximizer of $F$.
 \item\label{dc2}: For any $x$ if $x'\leftarrow \mathcal{C}(F,x)$ and $[x']$ is a discrete local minimizer of $F$ then \replaced{either $x'=x$, or $x''\leftarrow \mathcal{C}(f,[x'])$ implies $f([x''])\leq f(x^*)$}{$x'=x$ or being $x''\leftarrow \mathcal{C}(f,[x'])$, it holds that $f([x''])\leq f(x^*)$.}
\end{enumerate}
\end{definition}
The second condition prevents that the optimization procedure ends at points \replaced{where}{ such that} the value of $f$ increases.

\section{A filled function with respect to $\mathcal{C}$}
\label{sec:filledwrt}
\label{sec:3}
   Let $F(x^*,.): \mathbb{R}^n\rightarrow  \mathbb{R}$ \added{be }a discrete filled function of $f$ at $x^*$ (in the usual sense).
Define $\hat{F}(x^*,.): \mathbb{R}^n\rightarrow  \mathbb{R}$ as 
\begin{equation}
    \hat{F}(x^*,x)=F(x^*,x)+|F(x^*,x)|\sum_{i=1}^n sin^2(x_i\pi)
   \end{equation} \deleted{(in the definition the short notation $\hat{F}$ is used instead of $\hat{F}(x^*,.)$)}. Note that $F(x^*,x)=\hat{F}(x^*,x)$ for all $x\in\mathbb{Z}^n$. 
  
The function defined by the previous expression can be seen, informally as a wrapper for an existent filled function that made it suitable to be used with an arbitrary local search algorithm $\mathcal{C}$ (for example a continuous one) and additionally keeps the discrete nature of the problem.

The goal now is to prove that if $F$ is a discrete filled function of $f$, it can be translated into a filled function of $f$ with respect to $\mathcal{C}$.

\begin{theorem}
If $F$ is a discrete filled function of $f$ at $x^*$ and $B^*$ is the basin of $f$ at $x^*$, then 
$\hat{F}$ defined as above is a filled function of $f$ with respect to $\mathcal{C}$ at $x^*$.
\end{theorem}
\begin{proof}
If $F$ satisfies \ref{d1} then $\hat{F}$ trivially satisfies \ref{dc1} by the previous remark in the definition of $\hat{F}$.

	If $[x']$ is a discrete local minimizer of $F$, then it is also one of $\hat{F}$. By \ref{d2}, $[x']$ cannot be in a basin of $f$ higher than $B^*$, so if $x''\leftarrow \mathcal{C}(f,[x'])$ then $f([x'']) \leq f(x^*)$.  
\end{proof}
\subsection{An Additional Property of $\hat{F}$}
As the algorithm $\mathcal{C}$ may be a continuous algorithm, the computation of \replaced{$\mathcal{C}(\hat{F},x)$}{$\mathcal{C}(x,\hat{F})$} solves the problem of minimizing $\hat{F}$ without the integrality constraints. So is worthy to know the amount of error that the continuous relaxation of the problem introduces. The following proposition establishes an upper bound of that error. 

\begin{theorem}
 Let $x'_c$ be the point obtained in the \deleted{step} \ref{algo:s4}\ref{algo:s41} of the algorithm \ref{genericalgo} (before rounding), with the filled function $\hat{F}$. Let $\delta_i=\lfloor (x'_c)_i \rfloor - (x'_c)_i$ then 
 \[\sum_{i=1}^n \delta_i^2 < \frac{F(x^*)-F(x')}{4\mid F(x')\mid}.\] In particular, if $\hat{F}(x^*)=0$ then  $\sum_{i=1}^n \delta_i^2 < \frac{1}{4}$.
\label{theorem:approx}
\end{theorem}
\begin{proof}
 Since $x^*$ is a strict local maximizer of $\hat{F}$ then $\hat{F}(x^*+e_i)<\hat{F}(x^*)$ for all $i$. More over, by the assumption \ref{hipC1} over $\mathcal{C}$, $\hat{F}(x')<\hat{F}(x*)$. Because $x^*$ is in $\mathbb{Z}^n$, $\hat{F}(x^*)=F(x^*)$ and by the definition of $\hat{F}$ 
 \begin{equation}
 \sum_{i=1}^n sin^2(x_i\pi)<\frac{F(x^*)-F(x')}{\mid F(x')\mid} 
 \label{ineq:prop:1}
 \end{equation}

 It is well known that $\frac{2}{\pi}<\frac{sin(y)}{y}$ for $-\frac{\pi}{2}<y<\frac{\pi}{2}$. From $\mid\delta_i\mid<\frac{1}{2}$, it follows that $4\delta_i^2<\sum_{i=1}^n sin^2(\delta_i\pi)$. Finally, being $sin(x_i\pi)=sin(\delta_i\pi)$, the inequality (\ref{ineq:prop:1}) gives
 \[\sum_{i=1}^n 4\delta_i^2<\frac{F(x^*)-F(x')}{\mid F(x')\mid}\] and the result follows.
 
 In particular, if $\hat{F}(x^*)=F(x^*)=0$ then $F(x')<0$ so \[\sum_{i=1}^n 4\delta_i^2<\frac{-F(x')}{\mid F(x')\mid}=1.\]
 
\end{proof}
\section{Implementation and Numerical Results}
\label{sec:results}
\label{sec:4}
In the following we present a complete algorithm for the optimization of a discrete function using a filled function.
It allows the restart of algorithm \ref{genericalgo} from the best obtained point. Also, if there is no improvement between successive iterations,
an element in the discrete vicinity is chosen as the starting point for the next iteration.
The algorithm ends after the maximum number of iterations \added{$m$} is reached. 

\added{Usually, a very small $m$ (between $1$ and $3$) will be enough because every iteration is a restart of the optimization procedure from a different starting point. For that reason, the cycles that can appear in \ref{genericcompletealgo:2} by choosing points $x',x'',x',\ldots$ are of small length. Moreover, a check can be added to stop the process if a point $x'$ is reached more than $m'$ times (a user defined parameter).
}

\begin{enumerate}[label=Step \arabic*]
\label{genericcompletealgo}
\item Let $x_0$ be an initial point, and $m$ the maximum number of iterations.
Set to zero the counters $n_{fu}$, $n_{fill}$ for the evaluations of the original and the filled functions. Set $i\leftarrow 0$.
Set $x\leftarrow x_0$ as the current point and $x_g\leftarrow x_0$, $f_g\leftarrow f(x_g)$ as the best point and best value of $f$.
 \item Use the algorithm of Section \ref{sec:1} with $x$ as the starting point to obtain a minimizer $x'$ of $f$.
       Add the number of original and filled functions evaluations to the counters $n_{fu}$ and $n_{fill}$. 
       \label{genericcompletealgo:1}
  \item If $f(x')<f_g$, update $x_g\leftarrow x'$, $f_g\leftarrow f(x')$ and make the current point $x\leftarrow x_g$. 
  Else, choose a point $x''\added{\neq x'}$ in the discrete vicinity of $x'$ and make $x\leftarrow x''$.
  \label{genericcompletealgo:2}
 \item  Increment $i\leftarrow i+1$.
 \item If $i<m$ go to \ref{genericcompletealgo:1}. Else, the point $x_g$ is taken as the global minimizer.
 \label{genericcompletealgo:3}
\end{enumerate}
\subsection{Implementation}
The test code was written in FORTRAN 90 using software for continuous global optimization based on curvilinear searches
(see \cite{curvi}) as the algorithm $\mathcal{C}$ to perform the local search. 

\subsection{Results}

\replaced{The algorithm was tested on several small and moderate problems with a number of variables between $2$ and $100$ (functions of few variables were tested for comparison with the literature)}{The algorithm was tested with several nonlinear functions, three of which are of high dimensionality}, using different starting points. \deleted{For comparison with the literature, functions of few variables were also tested.} Four filled functions were used and are those that have already been used in \cite{criticalreview}. In all cases, each filled function
was modified according to section \ref{sec:filledwrt}. They are:
\begin{itemize}
 \item 	The filled function $1$, proposed in \cite{discretefilledng2}.
 \item 	The filled function $2$, proposed in \cite{discretefilledng1}.
 \item 	The filled function $3$, proposed in \cite{newfilledYang1}.
 \item 	The filled function $4$, proposed in \cite{newfilledYang2}.
\end{itemize}

 \subsubsection{Problem 1: Rosenbrock Function}
The Rosenbrock function is convex, multimodal and $n$-dimensional. The domain is usually taken to be $[-5,5]$, so the feasible region contains $11^n$ points. The unique global minimizer is $\bar{\textbf{x}}=(1,\ldots,1)$  with $f(\bar{\textbf{x}})=0$. The expression is
\[f(\textbf{x})=\sum_{i=1}^{n}[100 (x_{i+1} - x_i^2)^ 2 + (1 - x_i)^2]\]

The results for $n=50,100$ are shown in Table \ref{tabla:testfunresrosenbrock}.
\begin{table}[ht]
\caption{Results for the Rosenbrock function. FF is the number of the filled function employed. $f_g$ is the minimum reached, $n_{fu}$ and $n_{fill}$ are the number of function evaluations and the number of filled function evaluations}
 \label{tabla:testfunresrosenbrock}
 \begin{tabular}{lllllll}
  \hline
  $n$&Initial point&FF&$f_g$&$n_{fu}$&$n_{fill}$\\
  \hline
  $50$&$(3,3,\ldots,3)$&$  1$ & $0$ & $      138085$ & $     8093482$ \\
& & $2$ & $0$ & $      178486$ & $      161600$ \\
& & $3$ & $0$ & $       77186$ & $     1928195$ \\
& & $4$ & $0$ & $       26686$ & $      225563$ \\

  $100$&$(3,3,\ldots,3)$& $1$&$    0$&$         540817 $&$       56664532$\\
  &&$2$&$    0$&$          701617$&$          643200$\\
  &&$3$&$    0$&$          299017$&$        13639579$\\
  &&$4$&$    0$&$           98017$&$         1508952$\\

\end{tabular}
\end{table}
 \begin{table}[ht]
\caption{Comparison between the minimum number of function evaluations in 
\cite{discretefilledng2} and our results using the same filled function.}
\label{tabla:tesfuncomp}
\begin{tabular}{lll}
\hline
& Our results& Best results in \cite{discretefilledng2}\\
 $n$&$n_{fu}$& $n_{fu}$ \\
 \hline
$50$&$138085$&$1707270$\\
$100$& $540817 $& $13466632$ \\
\end{tabular}
\end{table}
\subsubsection{Problem 2: Rastrigin Function}
This function is convex, multimodal and has $n$ variables. It was evaluate in the region $[-5,5]$ and has $11^n$ feasible points. The unique global minimizer is $\bar{\textbf{x}}=(0,\ldots,0)$ with $f(\bar{\textbf{x}})=0$. The expression is 
\[f(\bar{\textbf{x}})=10n + \sum_{i=1}^{n}(x_i^2 - 10cos(2\pi x_i))\]
The results for $n=50,100$ are shown in Table \ref{tabla:testfunresrastrigin}.
\begin{table}[ht]
\caption{Results for the Rastrigin function with filled function 2. $f_g$ is the minimum reached, $n_{fu}$ and $n_{fill}$ are the number of function evaluations and the number of filled function evaluations. \deleted{$R_f$ is the ratio between the number of function evaluations (the objective plus the filled) and the size of the feasible set.}}
 \label{tabla:testfunresrastrigin}
 \begin{tabular}{lllll}
  \hline
  $n$&Initial point&$f_g$&$n_{fu}$&$n_{fill}$\\
  \hline
  $50$&$(-1,-1,\ldots,-1)$& $0$ & $    456714$ & $    414100$ \\
  &$(-5,5,\ldots)$&$0$& $    645398$ & $    434704$ \\\\
  $100$&$(-1,-1,\ldots,-1)$& $0$ & $2945914$&$2653200$\\
  &$(-5,5,\ldots)$&$0$ & $   4181432$ & $   2734002$ \\

\end{tabular}
\end{table}
 
\subsubsection{Other Functions}
 The \replaced{other}{others} test functions and the results are shown in the appendix. Their expressions, and their global minima are detailed in Tables \ref{tabla:testfunc1} and \ref{tabla:testfunc2}.
 The comparison with other results is given in Tables \ref{tabla:tesfuncompavg} and \ref{tabla:tesfuncomp}. 
 Table \ref{tabla:testfunres} shows the results for all the additional functions.
\section{Conclusions}
\label{sec:5}
To solve discrete nonlinear optimization problems is always a
challenging task. In this field, filled function methods have been proved to be
useful. Here, a more general approach to the filled functions methods has been 
introduced making them more suitable for being used with modern optimization algorithms. We also presented a way to move from standard definitions of the filled functions to the new one and introduced a new discrete filled function with the useful property that a good continuous global optimization algorithm applied to it leads to an approximation of the solution of the nonlinear discrete problem. The numerical results show the improvements over the usual approaches.
\section{Appendix}
\begin{longtable}{lll}
\caption{Additional test functions: names and expressions}\\
\label{tabla:testfunc1}
\endhead
\hline 
 1&Colville&    {$\begin{aligned}
           min \ f(\mathbf{x})&=100(x_2-x_1^2)^2+(1-x_1)^2+90(x_4-x_3^2)^2+(1-x_3)^2\\
           &+10.1((x_2-1)^2+(x_4-1)^2)+19.8(x_2-1)(x_4-1)\\
           \\
           s.t.\ & -10\leq x_i \leq 10,\ x_i \in\mathbb{Z}\ i=1,2,3,4\end{aligned}$}\\   
 \hline
 2&Goldstein and Price&{$\begin{aligned}
           min \ f(x,y)&=[1 + (x + y + 1)^2(19 - 14x+3x^2 - 14y + 6xy + 3y^2)]\\
           &[30 + (2x - 3y)^2(18 - 32x + 12x^2 + 4y - 36xy + 27y^2)]\\
           \\
           s.t.\ & x=\frac{z_1}{1000}\\
           &y=\frac{z_2}{1000}\\
           &-2000\leq z_i \leq 2000,\ z_i \in\mathbb{Z}\ i=1,2
          \end{aligned}$}\\
\hline
3&Beale&{$\begin{aligned}
     min\ f(x, y) &= (1.5-x+xy)^2+(2.25-x+xy^2)^2+(2.625-x+xy^3)^2\\
     \\
     s.t.\ & x=\frac{z_1}{1000}\\
           &y=\frac{z_2}{1000}\\
     &  -10000\leq z_i \leq 10000\\
     &z_i\in\mathbb{Z},\ i=1,2,
    \end{aligned}$}\\
\hline    
4&Powell singular&
{$\begin{aligned}
   min\ f(\mathbf{x}) &= f=(x_1+10x_2)^2+5(x_3-x_4)^2+(x_2-2x_3)^4+10(x_1-x_4)^4\\
     \\
     s.t.\ & x_i=\frac{z_i}{1000}\\
     &  -10000\leq z_i \leq 10000\\
     &z_i\in\mathbb{Z},\ i=1,2,3,4
  \end{aligned}
$}\\
\hline
 5&Booth&{$\begin{aligned}
min\ f(\mathbf{x})&=(x_1+2x_2-7)^2+(2x_1+x_2-5)^2\\
\\
s.t.\ &-10\leq x_i\leq 10,\ x_i\in\mathbb{Z},\ i=1,2
 \end{aligned}$}\\
 \hline
6& Problem 10 in \cite{discretefilledng2} & {$\begin{aligned}
min\ f(\mathbf{x})&=(x_1-1)^2+(x_2-1)^2+n \sum_{i=1}^{n-1}(n-i)(x_i^2-x_{i+1})^2\\
\\
s.t.\ &-5\leq x_i\leq 5,\ x_i\in\mathbb{Z},\ i=1,\ldots,n\\
&n=25\\
 \end{aligned}$}\\
 \hline
 7&Three-Hump Camel&{$\begin{aligned}
min\ f(\mathbf{x})&=2x_1^2-1.05x_1^4+\frac{x_1^6}{6}+x_1x_2+x_2^2\\
\\
s.t.\ &-5\leq x_i\leq 5,\ x_i\in\mathbb{Z},\ i=1,2
 \end{aligned}$}\\
\hline 
8&Schaffer N. 1&{$\begin{aligned}
min\ f(\mathbf{x})&=0.5 + \frac{sin^2(x_1^2+x_2^2)^2-0.5}{(1+0.001(x_1^2+x_2^2))^2}\\
\\
s.t.\ &-100\leq x_i\leq 100,\ x_i\in\mathbb{Z},\ i=1,2
 \end{aligned}$}\\
\hline 
9&Leon&
{$\begin{aligned}
min\ f(\mathbf{x})&= 100(x_2 - x_1^{3})^2 + (1 - x_1)^2\\
\\
s.t.\ &0\leq x_i\leq 10,\ x_i\in\mathbb{Z},\ i=1,2
 \end{aligned}$}\\
\hline
\\ 
10&Salomon&
{$\begin{aligned}
min\ f(\mathbf x)&=1-cos\Bigg(2\pi\sqrt{\sum_{i=1}^{n}x_i^2}\Bigg)+0.1\sqrt{\sum_{i=1}^{n}x_i^2}\\
\\
s.t.\ &-100\leq x_i\leq 100,\ x_i\in\mathbb{Z},\ i=1,
 \end{aligned}$}\\
 \hline
\end{longtable}
 \begin{table}[h]
\begin{minipage}{2in}
 \caption{Test functions, global minima}
\label{tabla:testfunc2}
\begin{tabular}{lll}
 \hline
 Function&$x^*_g$&$f(x^*_g)$\\
 \hline
 1&$(1,1,1,1)$ & $0$\\
 2& $(0,-1)$&$3$\\
 3&$(3,0.5)$&$0$ \\
 4&$(0,0,0,0)$&$0$\\
 5&$(1,3)$&$0$\\
 6&$(1,1,\ldots,1)$&$0$\\
 7&$(0,0)$&$0$\\
 8& $(0,0)$&$0$\\
 9& $(1,1)$&$0$ \\ 
 10& $(0,0)$&$0$ 
\end{tabular}
\end{minipage}
\begin{minipage}{2in}
 \caption{Comparison with the average number of original function evaluations in \cite{criticalreview} }
\label{tabla:tesfuncompavg}
\begin{tabular}{llll}
\hline
Function&FF&Our results&Avg. results in \cite{criticalreview}\\
\hline
1 & $1$    &   $3131$ &       $2440.17$\\
 & $ 2 $   &  $ 6085  $ &  $1679.5$\\
 & $ 3 $    &  $ 685  $ & $3430.5$\\
 & $ 4 $     &  $353  $ &    $2189.5$\\
2&$1$	  & $  983   $ &  $49533.17$\\
 &$2 $   & $ 1754   $ &  $22249$\\
 &$3 $    & $  238  $ &    $48327.17$\\
 &$4 $    &  $ 207  $ &  $46329.83$\\
3&$1 $    &  $1021  $ &    $366914.3$\\
 &$2 $    & $4237  $ & $119368.8$\\
 &$3 $    &  $ 281  $ & $1000001.5$\\
 &$4 $    &  $ 191  $ & $365956.2$\\
4&$1 $    &  $7156  $ & $1818$\\
 &$2 $    & $12455  $ &  $1123$\\
 &$3 $    & $ 1655  $ &   $2574.33$\\
 &$4 $    & $  963  $ &    $1811.83$\\
\end{tabular}
\end{minipage}
\end{table}
  \begin{table}
\caption{Comparison between the minimum number of function evaluations in 
\cite{criticalreview} and our results.}
\label{tabla:tesfuncomp}
\begin{tabular}{llllll}
\hline
Function Number&Our results& & Best results in \cite{criticalreview}& \\
 & $n_{fu}$&$n_{fill}$ & $n_{fu}$&$n_{fill}$ \\
 \hline
1&$353$&$711$&$1431$&$5099$\\
2&$200$&$644$&$21978$&$151356$&\\ 					 
3&$191$&$1620$&$100002$&$206268$\\
4&$963$&$8436$&$1179$&$5349$\\
\end{tabular}
\end{table}
 \begin{table}[ht]

\caption{Results. FF is the number of the filled function employed. $f_g$ is the minimum reached, $n_{fu}$ and $n_{fill}$ are the number of function evaluations and the number of filled function evaluations}
 \label{tabla:testfunres}
 
\begin{tabular}{llllll}
  \hline
  Function number&Initial point&FF&$f_g$&$n_{fu}$&$n_{fill}$\\
  \hline

$1$  &  $(0,0,0,0)$  &  $1$  &  $0$  & $3131$ & $26317$  \\

  &  $(0,0,0,0)$  &  $2$  &  $0$  & $6085$ & $5760$  \\

 &  $(0,0,0,0)$  &  $3$  &  $0$  & $685$ & $1032$  \\

  &  $(0,0,0,0)$  &  $4$  &  $0$  & $353$ & $711$  \\

$2$  &  $(1,-1)$  &  $1$  &  $3$  & $983$ & $8895$ \\

  &  $(1,-1)$  &  $2$  &  $3$  & $1747$ & $3538$ \\

  &  $(1,-1)$  &  $3$  &  $3$  & $231$ & $831$  \\

  &  $(1,-1)$  &  $4$  &  $3$  & $200$ & $644$ \\

$3$  &  $(0,0)$  &  $1$  &  $0$  & $1021$ & $1652$ \\

  &  $(0,0)$  &  $2$  &  $ 0.211400\cdot 10^{-4}$  & $4237$ & $4050$  \\

  &  $(0,0)$  &  $3$  &  $0$   & $281$ & $1819$  \\

  &  $(0,0)$  &  $4$  &  $0$   & $191$ & $1620$  \\

$4$  &  $(10,-10,10,-10)$  &  $1$  &  $0$   & $7156$ & $42924$  \\

  &  $(10,-10,10,-10)$  &  $2$  &  $0$   & $12455$ & $91212$  \\

  &  $(10,-10,10,-10)$  &  $3$  &  $0$   & $1655$ & $7842$  \\

  &  $(10,-10,10,-10)$  &  $4$  &  $0$   & $963$ & $8436$  \\

$5$  &  $(0,0)$  &  $1$  &  $0$   & $912$ & $3283$  \\

  &  $(0,0)$  &  $2$  &  $0$   & $1688$ & $1600$  \\

  &  $(0,0)$  &  $3$  &  $0$   & $172$ & $264$  \\

  &  $(0,0)$  &  $4$  &  $0$   & $88$ & $180$ \\

$6$  &  $(2,\ldots,2)$  &  $1$  &  $0$   & $331076$ & $3553422$  \\

  &  $(2,\ldots,2)$  &  $2$  &  $0$   & $622376$ & $612000$  \\

  &  $(2,\ldots,2)$  &  $3$  &  $0$   & $58793$ & $1661261$  \\

  &  $(2,\ldots,2)$  &  $4$  &  $0$   & $22372$ & $179670$ \\

$7$  &  $(2,2)$  &  $1$  &  $0$   & $6719$ & $95301$  \\

  &  $(2,2)$  &  $2$  &  $0.866667$   & $13671$ & $698488$ \\

  &  $(2,2)$  &  $3$  &  $0.866667$   & $3047$ & $20343$ \\

  &  $(2,2)$  &  $4$  &  $0.866667$   & $4903$ & $8963$  \\

$8$  &  $(-50,50)$  &  $1$  &  $0.370922$   & $4549$ & $169851$  \\
  &  $(-50,50)$  &  $2$  &  $0.487382$   & $6483$ & $318875$  \\

  &  $(-50,50)$  &  $3$  &  $0.489069$   & $1823$ & $26688$  \\

  &  $(-50,50)$  &  $4$  &  $0.487382$   & $2039$ & $5528$ \\

$9$  &  $(10,10)$  &  $1$  &  $0$   & $1183$ & $152490$  \\

  &  $(10,10)$  &  $2$  &  $0$   & $1267$ & $600$  \\

  &  $(10,10)$  &  $3$  &  $0$   & $781$ & $348$ \\

 &  $(10,10)$  &  $4$  &  $0$   & $673$ & $302$ \\

$10$  &  $(-100,100,-100,\ldots)$  &  $1$  &  $0$   & $11818$ & $208250$  \\

  &  $(-100,100,-100,\ldots)$  &  $2$  &  $0$   & $20767$ & $186923$  \\

  &  $(-100,100,-100,\ldots)$  &  $3$  &  $1.5$   & $4580$ & $19432$  \\

  &  $(-100,100,-100,\ldots)$  &  $4$  &  $0$   & $2275$ & $2709$  \\

\end{tabular}

\end{table}
 
 \medskip

\end{document}